\documentclass[11pt]{article}
\usepackage{amsthm,amsmath,amssymb}
\newcommand{\subject}[1]{\begin{flushleft}\textbf{2010 Mathematics  Subject Classification}: #1\end{flushleft}}
\newcommand{\keyword}[1]{\par\noindent \textbf{Keywords:} #1 }
\newcommand{\university}[1]{\\[3mm]{\small #1}}

\newtheorem{defi}{Definition}[section]

\newtheorem{theorem}[defi]{Theorem}

\title{Symmetry Reduction of the Two-Dimensional Ricci Flow Equation }
\author{Mehdi Nadjafikhah
\thanks{Corresponding author. School of Mathematics, Iran University of Science and Technology, Narmak, Tehran 1684613114, Iran. E-mail: m\underline{ }nadjafikhah@iust.ac.ir}
\university{}
 \and
 Mehdi Jafari
 \thanks{Department of Complementary Education, Payame Noor University, PO BOX 19395-3697, Tehran, Iran. E-mail: m.jafari@phd.pnu.ac.ir}
 \university{} }
\date{\today}

\begin{document}
\maketitle
\begin{abstract}
This paper is devoted to obtain the one-dimensional group invariant solutions of the two-dimensional Ricci flow ($(2D)$ Rf) equation. By classifying the orbits of the adjoint representation, the optimal system of one-dimensional subalgebras of the  ($(2D)$ Rf) equation is obtained. For each class, we will find the reduced equation by method of similarity reduction. By solving these reduced equations we will obtain new sets of group invariant solutions for the ($(2D)$ Rf) equation.

 \keyword{Lie symmetry group, two-dimensional Ricci flow equation, optimal system, group invariant solution}
 \subject{70G65, 34C14, 53C44}
\end{abstract}
\section{Introduction}
The Ricci flow was introduced by Hamilton in his seminal paper, ``Three-manifolds with positive Ricci curvature" in 1982 \cite{Ham1}. Since then, Ricci flow has been a very useful tool for studying the special geometries which a manifold admits. Ricci flow is an evolution equation for a Riemannian metric which sometimes can be used in order to deform an arbitrary metric into a suitable metric that can specify the topology of the underlying manifold. If $(M,g(t))$ be a smooth Riemannian manifold, Ricci flow is defined by the equation
\begin{eqnarray}
\frac{\partial}{\partial t}g(t)=-2Ric,
\end{eqnarray}
where $Ric$ denotes the Ricci tensor of the metric $g$.
By using the concept of Ricci flow, Grisha Perelman completely proved the Poincar$\acute{\mbox{e}}$ conjecture around 2003 [12-14].
The Ricci flow also is used as an approximation to the renormalization group flow for the two-dimensional nonlinear $\sigma$-model, in quantum field theory. see \cite{Gaw} and references therein.

In this paper we want to obtain new solutions of the ($(2D)$ Rf) equation by method of Lie symmetry group. As it is well known, the Lie symmetry group method has an important role in the analysis of differential equations. The theory of Lie symmetry groups of differential equations was developed by Sophus Lie \cite{Lie}. By this method we can reduce the order of ODEs and investigate the invariant solutions. Also we can construct  new solutions from known ones  (for more details about the applications of Lie symmetries, see \cite{Olv,Blucol,BluKum}).
Lei's method led to an algorithmic approach to find special solution of differential equation by its symmetry group. These solutions are called group invariant solutions and  obtain by solving the reduced system of differential equation having fewer independent variables than the original system. Bluman and Cole generalized the Lie's symmetry method for finding the group-invariant solutions \cite{BluCol}.
In this paper we apply this method to obtain the invariant solutions of ($(2D)$ Rf) equation and classify them.

This paper is organized as follows. In section 2, by using the mechanical model of Ricci flow, Lie symmetries of ($(2D)$ Rf) equation will be state and some results yield from the structure of the Lie algebra of the Lie symmetry group. In section 3, we will construct an optimal system of one-dimensional subalgebras of the ($(2D)$ Rf) equation which is useful for classifying of group invariant solutions. In section 4, the reduced equation for each element of optimal system is obtained. In section 5, we will solve the reduced equations by method of Lie symmetry group and obtain the group invariant solutions of ($(2D)$ Rf) equation.

\section{Lie Symmetries of ($(2D)$ Rf) Equation}
As we know, transformations which map solutions of a differential equation to other solutions are called symmetries of the equation. In \cite{Olv} Olver has introduced the procedure of finding Lie symmetry group of a differential  equation. Cimpoiaus and Constantinescu have introduced the mechanical model for the Ricci flow as follows:
 \begin{eqnarray}\label{2}
u^2u_t+u_yu_x-uu_{xy}=0
\end{eqnarray}
and obtained the Lie symmetry group of this equation \cite{Cim}. They proved that this equation admits a 6-parameter Lie group, $G$, with the following infinitesimal generators for its Lie algebra, $\mathfrak{g}$.
 \begin{eqnarray}\label{3}
\begin{array}{lclclclclc}
X_1=\partial_x, &&X_2=\partial_y, && X_3=\partial_t,\\
 X_4=t\partial_t+u\partial_u, &&X_5=x\partial_x-u\partial_u, &&X_6=y\partial_y-u\partial_u.
\end{array}
\end{eqnarray}
The commutator table of Lie algebra for  $\mathfrak{g}$ is given below, where the entry in the $i^{th}$ row and $j^{th}$ column is $[X_{i},X_{j}]=X_{i}X_{j}-X_{j}X_{i}$, $i,j=1,...,6$.
\begin{table}[h]
\begin{center}
{\small \textbf{Table 1:} The commutator table of $\mathfrak{g}$.}

\begin{tabular}{ccccccccccccc}
\cline{1-13}
${[\,,\,]}$ && $X_1$  && $X_2$    && $X_3$    && $X_4$  && $X_5$ && $X_6$  \\
\cline{1-13}
$X_1$ && $0$      && $0$    && $0$      && $0$     && $X_1$  && $0$   \\
$X_2$ && $0$      && $0$    && $0$      && $0$     && $0$  && $X_2$   \\
$X_3$ && $0$      && $0$     && $0$     && $X_3$  && $0$ && $0$   \\
$X_4$ && $0$      && $0$     && $-X_3$ && $0$     && $0$ && $0$   \\
$X_5$ && $-X_1$  && $0$     && $0$     && $0$     && $0$&&  $0$  \\
$X_6$ && $0$       && $-X_2$&& $0$     && $0$     && $0$ && $0$  \\
\cline{1-13}
\end{tabular}
\end{center}
\end{table}

Exponentiating the infinitesimal symmetries (\ref{3}) we obtain the one-parameter groups $g_k(s)$ generated by $X_k$, $k=1,...,6$:
 \begin{eqnarray}
\begin{array}{lclcl}
g_1(s)&:&(x,y,t,u)&\longmapsto&(x+s,y,t,u),\\ g_2(s)&:&(x,y,t,u)&\longmapsto&(x,y+s,t,u),\\ g_3(s)&:&(x,y,t,u)&\longmapsto&(x,y,t+s,u),\\ g_4(s)&:&(x,y,t,u)&\longmapsto&(x,y,te^s,ue^s),\\ g_5(s)&:&(x,y,t,u)&\longmapsto&(xe^s,y,t,ue^{-s}),\\ g_6(s)&:&(x,y,t,u)&\longmapsto&(x,ye^s,t,ue^{-s}),\\ \end{array}
\end{eqnarray}
Consequently, we can state the following theorem:
\begin{theorem}
If $f=f(x,y,t)$ is a solution of (\ref{2}), so are functions
 \begin{eqnarray}
\begin{array}{lclclcl}
g_1(s).f=f(x-s,y,t), && g_4(s).f=f(x,y,te^{-s})e^s,\\
g_2(s).f=f(x,y-s,t), && g_5(s).f=f(xe^{-s},y,t)e^{-s},\\
g_3(s).f=f(x,y,t-s), && g_6(s).f=f(x,ye^{-s},t)e^{-s}.
 \end{array}\end{eqnarray}
\end{theorem}

\section{One-dimensional optimal system of subalgebras for the (($2$D) Rf) equation }
In this section, we obtain the one-dimensional optimal system of (($2$D) Rf) equation by using symmetry group. Since every linear combination of infinitesimal symmetries is an infinitesimal symmetry, there is an infinite number of one-dimensional subgroups for $G$. Therefore,  it's important to determine which subgroups give different types of solutions. For this, we must find invariant solutions which can not be transformed to each other  by symmetry transformations in the full symmetry group. This led to concept of an optimal system of subalgebra. For one-dimensional subalgebras, this classification problem is the same as the problem of classifying the orbits of the adjoint representation \cite{Olv}. Optimal set of subalgebras is obtaining from taking only one representative from each class of equivalent subalgebras. The  problem of classifying the orbits is solved by taking a general element in the lie algebra and simplify it as much as possible by imposing various adjoint transformation on it \cite{Ovs, Nad2}.
Adjoint representation of each $X_i$, $i=1,...,6$ is defined as follow:
\begin{equation}
\mbox{Ad}(\exp(s.X_i).X_j) =
X_j-s.[X_i,X_j]+\frac{s^2}{2}.[X_i, [X_i,X_j]]-\cdots,
\end{equation}
where $s$ is a parameter and  $[X_i,X_j]$  is the commutator of the Lie algebra for $i,j=1,\cdots,6$ (\cite{Olv},page 199).
Taking into account the table of commutator, we can compute all the adjoint representations corresponding to the Lie group of the (($2$D) Rf) equation. They are presented in Table 2. Not that, the $(i,j)$ entry indicate $\mbox{Ad}(\exp(s.X_i).X_j)$.

\begin{table}[h]
\begin{center}
{\small \textbf{Table 2:} The adjoint representation table of the infinitesimal generators $X_i$.}

\begin{tabular}{ccccccccccccc}
\cline{1-13}
$\mbox{Ad}$ && $X_1$  && $X_2$    && $X_3$    && $X_4$  && $X_5$ && $X_6$  \\
\cline{1-13}
$X_1$ && $X_1$      && $X_2$    && $X_3$      && $X_4$     && $X_5-sX_1$  && $X_6$   \\
$X_2$ && $X_1$      && $X_2$    && $X_3$      && $X_4$     && $X_5$  && $X_6-sX_2$   \\
$X_3$ && $X_1$      && $X_2$     && $X_3$     && $X_4-sX_3$  && $X_5$ && $X_6$   \\
$X_4$ && $X_1$      && $X_2$     && $e^sX_3$ && $X_4$     && $X_5$ && $X_6$   \\
$X_5$ && $e^sX_1$  && $X_2$     && $X_3$     && $X_4$     && $X_5$&&  $X_6$  \\
$X_6$ && $X_1$       && $e^sX_2$&& $X_3$     && $X_4$     && $X_5$ && $X_6$  \\
\cline{1-13}
\end{tabular}
\end{center}
\end{table}
Now we can state the following theorem:
\begin{theorem}\label{3.1}
A one-dimensional optimal system for Lie algebra of (($2$D) Rf) equation is given by
\begin{eqnarray}\label{7}
\begin{array}{lclcl}
1) X_4+aX_5+bX_6,&& 4) \varepsilon X_1+\varepsilon ' X_2+X_4,&& 7) \varepsilon X_1+ \varepsilon ' X_3+X_6,\\
2) \varepsilon X_2+X_4+aX_5,&& 5) \varepsilon X_3+X_5+aX_6, && 8) \varepsilon X_1+cX_2+ \varepsilon ' X_3,\\
3) \varepsilon X_1+X_4+aX_6,&& 6) \varepsilon X_2+\varepsilon ' X_3+X_5,\\
\end{array}
\end{eqnarray}
where  $\varepsilon$ and $ \varepsilon '$ are $\pm 1$ or zero. Also  $a,b,c\in\mathbb{R}$ and $a\neq 0, b\neq 0$.
\end{theorem}
\begin{proof}
Let $F^s_i:\mathfrak{g}\to \mathfrak{g}$ be the adjoint transformation defined by
$X\mapsto\mbox{Ad}(\exp(sX_i).X)$, for
$i=1,\cdots,6$. The matrix of $F^s_i$, $i=1,\cdots,6$, with
respect to basis $\{X_1,\cdots,X_6\}$ is:
\begin{eqnarray*}
\begin{small}
\begin{array}{lclclclclclc}
 M_1^s\!=\!\!\!\left[ \begin{array}{cccccc}
1&0&0&0&-s&0\\0&1&0&0&0&0\\0&0&1&0&0&0\\0&0&0&1&0&0\\0&0&0&0&1&0\\0&0&0&0&0&1
\end{array} \right],&
 M_2^s\!=\!\!\!\left[ \begin{array}{cccccc}
1&0&0&0&0&0\\0&1&0&0&0&-s\\0&0&1&0&0&0\\0&0&0&1&0&0\\0&0&0&0&1&0\\0&0&0&0&0&1
\end{array} \right],&
M_3^s\!=\!\!\!\left[ \begin{array}{cccccc}
1&0&0&0&0&0\\0&1&0&0&0&0\\0&0&1&-s&0&0\\0&0&0&1&0&0\\0&0&0&0&1&0\\0&0&0&0&0&1
\end{array} \right],

\vspace*{5mm}

\\

\vspace*{2mm}

M_4^s\!=\!\!\!\left[ \begin{array}{cccccc}
1&0&0&0&0&0\\0&1&0&0&0&0\\0&0&e^s&0&0&0\\0&0&0&1&0&0\\0&0&0&0&1&0\\0&0&0&0&0&1
\end{array} \right],&
M_5^s\!=\!\!\!\left[ \begin{array}{cccccc}
e^s&0&0&0&0&0\\0&1&0&0&0&0\\0&0&1&0&0&0\\0&0&0&1&0&0\\0&0&0&0&1&0\\0&0&0&0&0&1
\end{array} \right],
&
M_6^s\!=\!\!\!\left[ \begin{array}{cccccc}
1&0&0&0&0&0\\0&e^s&0&0&0&0\\0&0&1&0&0&0\\0&0&0&1&0&0\\0&0&0&0&1&0\\0&0&0&0&0&1
\end{array} \right].
\end{array}
\end{small}
\end{eqnarray*}
respectively. If $X=\sum_{i=1}^6a_iX_i$, then we have
\begin{eqnarray*}
\begin{array}{lclclclclc}
&F^{s_6}_6\circ F^{s_5}_5\circ \cdots \circ F^{s_1}_1& \ : \ X \mapsto \ &e^{s_5}a_1X_1+e^{s_6}a_2X_2+e^{s_4}a_3X_3+(a_4-e^{s_4}s_3a_3)X_4+

 \\& & &\hspace{-22mm}(a_5-e^{s_5}s_1a_1)X_5+(a_6-e^{s_6}s_2a_2)X_6.&
\end{array}\end{eqnarray*}
Now, we try to vanish the coefficients of $X$ by acting the adjoint representations $M^{s_i}_i$ on $X$, by choosing suitable parameters $s_i$ in each step. Therefor we can simplify $X$ as follows:
\begin{itemize}
\item[] If $a_4\neq 0$, $a_5\neq 0$ and $a_6\neq 0$, we can make the coefficients of $X_3$, $X_1$ and $X_2$ vanish by $F^{S_3}_3$ , $F^{S_1}_1$ and $F^{S_2}_2$. By setting $s_3=-\frac{a_3}{a_4}$, $s_1=-\frac{a_1}{a_5}$ and $s_2=-\frac{a_2}{a_6}$ respectively. Scaling $X$ if necessary, we can assume that   $a_4=1$. So $X$ is reduced to the case  (1).

\item[] If $a_4\neq 0$, $a_5\neq 0$ and $a_6=0$, we can make the coefficients of $X_3$ and  $X_1$ vanish by $F^{S_3}_3$ and $F^{S_1}_1$; By setting $s_3=-\frac{a_3}{a_4}$ and $s_1=-\frac{a_1}{a_5}$, respectively. Also by setting $s_6=\mbox{ln}\vert a_2\vert$ in $F^{S_6}_6$, we can make the coefficient of $X_2$ vanish or $\pm 1$. Scaling $X$ if necessary, we can assume that   $a_4=1$. So $X$ is reduced to the case  (2).

\item[] If $a_4\neq 0$, $a_6\neq 0$ and $a_5=0$, we can make the coefficients of $X_3$ and  $X_2$ vanish by $F^{S_3}_3$ and $F^{S_2}_2$; By setting $s_3=-\frac{a_3}{a_4}$ and $s_2=-\frac{a_2}{a_6}$, respectively. Also by setting $s_5=\mbox{ln}\vert a_1\vert$ in $F^{S_5}_5$, we can make the coefficient of $X_1$ vanish or $\pm 1$. Scaling $X$ if necessary, we can assume that   $a_4=1$. So $X$ is reduced to the case (3).

\item[] If $a_4\neq 0$ and $a_5=a_6=0$, we can make the coefficient of $X_3$ vanish by $F^{S_3}_3$; By setting $s_3=-\frac{a_3}{a_4}$. Also by setting $s_5=\mbox{ln}\vert a_1\vert$ and $s_6=\mbox{ln}\vert a_2\vert$ in $F^{S_5}_5$ and $F^{S_6}_6$, we can make the coefficients of $X_1$ and $X_2$ vanish or $\pm 1$, respectively. Scaling $X$ if necessary, we can assume that   $a_4=1$. So $X$ is reduced to the case  (4).

\item[] If $a_5\neq0$, $a_6\neq 0$ and $a_4=0$, we can make the coefficients of $X_1$ and  $X_2$ vanish by $F^{S_1}_1$ and $F^{S_2}_2$; By setting $s_1=-\frac{a_1}{a_5}$ and $s_2=-\frac{a_2}{a_6}$, respectively. Also by setting $s_4=\mbox{ln}\vert a_3\vert$ in $F^{S_4}_4$, we can make the coefficient of $X_1$ vanish or $\pm 1$. Scaling $X$ if necessary, we can assume that   $a_5=1$. So $X$ is reduced to the case  (5).

\item[] If $a_5\neq 0$ and $a_4=a_6=0$, we can make the coefficient of $X_1$ vanish by $F^{S_1}_1$; By setting $s_1=-\frac{a_1}{a_5}$. Also by setting $s_6=\mbox{ln}\vert a_2\vert$ and $s_4=\mbox{ln}\vert a_3\vert$ in $F^{S_6}_6$ and $F^{S_4}_4$, we can make the coefficients of $X_2$ and $X_3$ vanish or $\pm 1$, respectively. Scaling $X$ if necessary, we can assume that   $a_5=1$. So $X$ is reduced to the case (6).

\item[] If $a_4=a_5=0$ and $a_6\neq 0$, we can make the coefficient of $X_2$ vanish by  $F^{S_2}_2$; By setting $s_2=-\frac{a_2}{a_6}$. Also by setting $s_5=\mbox{ln}\vert a_1\vert$ and $s_4=\mbox{ln}\vert a_3\vert$ in $F^{S_5}_5$ and $F^{S_4}_4$, we can make the coefficients of $X_1$  and $X_3$ vanish or $\pm 1$, respectively. Scaling $X$ if necessary, we can assume that   $a_6=1$. So $X$ is reduced to the case  (7).

\item[] If $a_4=a_5=a_6=0$, we can make the coefficients of $X_1$  and $X_3$ vanish or $\pm 1$ by $F^{S_5}_5$ and $F^{S_4}_4$; By setting $s_5=\mbox{ln}\vert a_1\vert$ and $s_4=\mbox{ln}\vert a_3\vert$, respectively.  So $X$ is reduced to the case (8).
\end{itemize}
\end{proof}
\section{Similarity Reduction of  (($2$D) Rf) Equation}
In this section, the two-dimensional Ricci flow equation will be reduced by expressing it in the new coordinates. The (($2$D) Rf) equation is expressed in the coordinates $(x,y,t,u)$, we must search for this equation's form in the suitable coordinates for reducing it. These new coordinates will be obtained by looking for  independent invariants $(z,w,f)$ corresponding to the generators of the symmetry group. Hence, By using the new coordinates and applying the chain rule, we obtain the reduced  equation.  We express this procedure for one of the infinitesimal generators in the optimal system $(\ref{7})$ and list the result for some other cases.

For example, consider the case (7) in theorem $(\ref{3.1})$ when $\varepsilon=1$ and $\varepsilon^{'}=0$, therefor we have $X:=X_1+X_6$. For determining independent invariants $I$, we ought to solve the PDEs $X(I)=0$, that is
\begin{equation}
(X_1+X_6)I=(\partial_x+y\partial_y-u\partial_u)I=\frac{\partial I}{\partial x}+y\frac{\partial I}{\partial y}+0\frac{\partial I}{\partial t}-u\frac{\partial I}{\partial u}=0.
\end{equation}
For solving this PDE, the following associated characteristic ODE must be solved:
\begin{equation}
\frac{dx}{1}=\frac{dy}{y}=\frac{dt}{0}=\frac{du}{-u}.
\end{equation}
Hence, three functionally independent invariants  $z=ye^{-x}$, $w=t$ and $f=uy$ are obtained. If we treat $f$ as a function of $z$ and $w$, we can compute formulae for the derivatives of $u$ with respect to   $x$, $y$ and $t$ in terms of $z$, $w$, $f$ and the derivatives of $f$ with respect to $z$ and $w$. By using the chain rule and the fact that $u=f(z,w)y^{-1}$ we have
 \begin{eqnarray}
\begin{array}{lclclclclc}
u_t=(f_zz_t+f_ww_t)y^{-1}=f_wy^{-1}, && u_x=-f_ze^{-x},\\
 u_y=f_ze^{-x}y^{-1}-fy^{-2}, &&u_{xy}=-e^{-2x}f_{zz}.
\end{array}
\end{eqnarray}
After substituting the above relations into the  equation (\ref{2}), we obtain:
 \begin{eqnarray}
\begin{array}{lclclclclc}
u^2u_t+u_yu_x-uu_{xy}=y^{-3}(f^2f_w-f^2_zz^2+ff_zz+ff_{zz}z^2)=0.
\end{array}
\end{eqnarray}
So the reduced equation is
 \begin{eqnarray}\label{12}
f^2f_w-z^2f^2_z+zff_z+z^2ff_{zz}=0.
\end{eqnarray}
This equation has two independent variables $z$ and $w$ and one dependent variable $f$. In a similar way, we can compute all of the similarity reduction equations corresponding to the infinitesimal symmetries that mentioned in theorem $(\ref{3.1})$. some of them are listed in Table 3.
\begin{table}[h]
\begin{small}\begin{center}
{\small \textbf{Table 3:} Lie invariants, similarity solutions and reduced equations.}

\begin{tabular}{lllllll}
\cline{1-5}
$i$ & $\mathfrak{h}_i$  & $\{z_i,w_i,v_i\}$   & $u_i$ & Similarity reduced equations \\
\cline{1-5}
$1$ & $X_1+X_6$ & $\{ye^{-x},t,uy\}$ & $\frac{f(z,w)}{y}$ & $f^2f_w-z^2f^2_z+zff_z+z^2ff_{zz}=0$
\vspace*{2mm}  \\
$2$ & $X_2+X_4$ & $\{x,te^{-y},ue^{-y}\}$ & $f(z,w)e^{y}$ & $f^2f_w-wf_zf_w+wff_{zw}=0$
\vspace*{2mm}  \\
$3$ & $X_3+X_5+aX_6$      & $\{\frac{y}{x^a},\mbox{ln}\frac{e^t}{x},ux^{a+1}\}$     & $\frac{f(z,w)}{x^{a+1}}$ & $f_w(f^2-f_z)-azf^{2}_z+f(af_z+azf_{zz}+f_{zw})=0$ \vspace*{2mm}  \\
$4$ & $X_5+aX_6$      & $\{\frac{y}{x^a},t,ux^{a+1}\}$  & $\frac{f(z,w)}{x^{a+1}}$ & $f^2f_w+a(ff_z-zf^{2}_z+zff_{zz})=0$
\vspace*{2mm}   \\
$5$ & $X_2+X_3+X_5$      & $\{\mbox{ln}\frac{e^y}{x},\mbox{ln}\frac{e^t}{x},ux\}$  & $\frac{f(z,w)}{x}$ & $f^2f_w-f^{2}_z-f_wf_z+ff_{zz}+ff_{zw}=0$
 \vspace*{2mm}\\
$6$ & $X_2+X_5$  & $\{\mbox{ln}\frac{e^y}{x},t,ux\}$     & $\frac{f(z,w)}{x}$
&$f^2f_w-f^2_z+ff_{zz}=0$
\vspace*{2mm}  \\
$7$ & $-X_3+X_6$       &$\{x,t+\mbox{ln}y,uy\}$     & $\frac{f(z,w)}{y}$  &$f^2f_w+f_zf_w-ff_{zw}=0$
 \vspace*{2mm} \\
$8$ & $X_1+aX_2$       & $\{y-ax,t,u\}$     & $f(z,w)$ &$f^2f_w-af^{2}_z+aff_{zz}=0$
 \vspace*{2mm}  \\
$9$ & $aX_2+X_3$       & $\{x,\frac{ta-y}{a},u\}$     & $f(z,w)$
&$af^2f_w-f_zf_w+ff_{zw}=0$
\vspace*{2mm} \\
\cline{1-5}
\end{tabular} \\[5mm]
\end{center}\end{small}
\end{table}

\section{Group Invariant Solutions of (($2$D) Rf) Equation }
In this section we  reduce the equations obtained in last section to ODEs and solve them.

For example, the equation $(\ref{12})$ admits a 4-parameters family of Lie operators with following infinitesimal generators
 \begin{eqnarray}
\begin{array}{lclclclclc}
V_1=\frac{1}{2}z\mbox{ln}z\partial_z+w\partial_w, && V_3=-\frac{1}{2}z\mbox{ln}z\partial_z+f\partial_f,\\
V_2=\partial_w, && V_4=z\partial_z.
\end{array}
\end{eqnarray}
The invariants associated to the infinitesimal generator $V_2$, are $s=z$ and $g=f$. By substituting these invariants into the equation $(\ref{12})$ and using chain rule, the reduced equation is obtained as follows:
 \begin{eqnarray}
sg'^2-gg'-sgg''=0
\end{eqnarray}
the solution of this equation is $g(s)= c_2s^{c_1}=c_2z^{c_1}$. therefore we have $f(z)=c_2z^{c_1}=c_2(ye^{-x})^{c_1}=c_2y^{c_1}e^{-c_1x}$. So $u=fy^{-1}=c_2y^{c_1-1}e^{-c_1x}$ is a solution of (($2$D) Rf) equation.

By similar arguments, we can obtain other invariant solutions of the equation $(\ref{12})$. Also by reducing other equations in Table 3, we can find other solutions of  (($2$D) Rf) equation. Some of the similarity reduced equations and their invariant solutions are listed in Table 4 and Table 5 respectively.

\begin{table}[h]
\begin{small}
\begin{center}
{\small \textbf{Table 4:} ODEs that obtain from the reduced equations of table 3. }

\begin{tabular}{llllllllll}
\cline{1-5}
$i$ & $\begin{array}{lcr} \mbox{ Simmetry\ group}\\ \mbox{generators} \end{array}$   & $\begin{array}{lcr} \mbox{Optimal}\\ \mbox{system} \end{array}$  & $\begin{array}{lcr} \mbox{Invariants}\\ \{s,g\} \end{array}$ &  $\begin{array}{lcr} \mbox{Reduced}\\ \mbox{equation} \end{array}$  \\
\cline{1-5}
$1$ & $  \begin{array}{lcr} V_1=\frac{1}{2}z\mbox{ln}z\partial_z+w\partial_w\\ V_2=\partial_w  \\ V_3=-\frac{1}{2}z\mbox{ln}z\partial_z+f\partial_f \\ V_4=z\partial_z \\ \end{array} $      & $\begin{array}{lcr} \mathcal{A}^1_1:V_2 \\ \mathcal{A}^2_1:V_3 \\ \mathcal{A}^3_1:V_1+V_3 \\ \mathcal{A}^4_1:V_2+V_4 \end{array}$  & $  \begin{array}{lcr}\{z,f\}\\ \{w,f(\mbox{ln}z)^2\}\\ \{z,\frac{f}{w}\}\\ \{w-\mbox{ln}z,f\}\end{array} $ & $  \begin{array}{lcr}sg'^2-gg'-sgg''=0\\ g^2(g'+2)=0\\g^3-s^2g'^2+sgg'+s^2gg''=0\\g^2g'-g'^2+gg''=0   \end{array} $
 \vspace*{2mm}
 \\

$2$ & $  \begin{array}{lcr} V_1=w \partial_w+f\partial_f\\ V_2=z\partial_z-f\partial_f \\ V_3=\partial_z \\ \end{array} $      & $\begin{array}{lcr} \mathcal{A}^1_2:V_2 \\ \mathcal{A}^2_2:V_3  \end{array}$     & $  \begin{array}{lcr} \{w,fz\}\\ \{w,f\}\end{array} $ &$ \begin{array}{lcr}g^2g'=0\\g^2g'=0 \end{array} $
 \vspace*{2mm}
  \\

 $3$ & $  \begin{array}{lcr} V_1=z\mbox{ln}z\partial_z+w\partial_w-\\ \hspace*{8mm}f(1+\mbox{ln}z)\partial_f\\ V_2=\partial_w \\ V_3=z\partial_z-f\partial_f \\ \end{array} $      & $\begin{array}{lcr} \mathcal{A}^1_3:V_1\\ \\ \mathcal{A}^2_3:V_2 \\ \mathcal{A}^3_3:V_3  \end{array}$     & $  \begin{array}{lcr} \{\frac{w}{\mbox{ln}z},fz\mbox{ln}z\}\\ \\ \{z,f\}\\ \{w,fz\}\end{array} $ &$ \begin{array}{lcr}gg'(g+2as-1)+ag^2+\\ s(as-1)(gg''-g'^2)=0\\ sg'^2-gg'-sgg''=0\\ g^2g'=0 \end{array}$
 \vspace*{2mm}
  \\

 $4$ & $  \begin{array}{lcr} V_1=z \partial_z+w\partial_w\\ V_2=\partial_w\\ V_3=-z \partial_z+f\partial_f\\ V_4=z(2-\mbox{ln}z)\partial_z+\\ \hspace*{8mm}f\mbox{ln}z \partial_f \\ \end{array} $      & $\begin{array}{lcr} \mathcal{A}^1_4:V_4 \\ \\ \mathcal{A}^2_4:V_1+V_3  \end{array}$     & $  \begin{array}{lcr} \{w,fz(4-\\ \mbox{ln}z(4-\mbox{ln}z))\}\\ \{z,\frac{f}{w}\}\end{array} $ &$ \begin{array}{lcr}g^2(g'+2a)=0\\ \\g^3-a(sg'^2+gg'+sgg'')=0 \end{array} $
 \vspace*{2mm}
  \\

 $5$ & $  \begin{array}{lcr} V_1=z\partial_z+w \partial_w-f\partial_f\\ V_2=\partial_w \\ V_3=\partial_z \\ \end{array} $      & $\begin{array}{lcr} \mathcal{A}^1_5:V_1 \\ \\ \mathcal{A}^2_5:V_2\\ \mathcal{A}^3_5:V_2-V_3 \end{array}$     & $  \begin{array}{lcr} \{\frac{w}{z},fz\}\\ \\ \{z,f\} \\ \{z+w,f\}\end{array} $ &$ \begin{array}{lcr}gg'(1-g-2s)-g^2+\\s(s-1)(g'^2-gg'')=0\\-g'^2+gg''=0\\g^2g'-2g'^2+2sgg''=0 \end{array} $
 \vspace*{2mm}
   \\

 $6$ & $  \begin{array}{lcr} V_1=\frac{z}{2} \partial_z+w\partial_w\\ V_2=\partial_w \\ V_3=-\frac{z}{2}\partial_z+f\partial_f\\ V_4=\partial_z \\ \end{array} $      & $\begin{array}{lcr}  \mathcal{A}^1_6:V_3\\ \mathcal{A}^2_6:V_1+V_3\\ \mathcal{A}^3_6:V_2+V_4  \end{array}$     & $  \begin{array}{lcr} \{w,fz^2\}\\ \{z,\frac{f}{w}\}\\ \{w-z,f\}\end{array} $ &$ \begin{array}{lcr}g'+2=0 \\ g^3-g'^2+gg''=0\\ g^2g'-g'^2+gg''=0\end{array} $
 \vspace*{2mm}
  \\

 $7$ & $  \begin{array}{lcr} V_1=w \partial_w\\ V_2=\partial_w\\V_3=z\partial_z-f\partial_f \\ V_4=\partial_z \\ \end{array} $      & $\begin{array}{lcr}  \mathcal{A}^1_7:V_4  \\ \mathcal{A}^2_7:V_2+V_3\end{array}$     & $  \begin{array}{lcr} \{w,f\}\\  \{w-\mbox{ln}z,zf\}\end{array} $ &$ \begin{array}{lcr}g'=0 \\ g^2g'-g'^2+gg''=0 \end{array} $
 \vspace*{2mm}
 \\

 $8$ & $  \begin{array}{lcr} V_1=\frac{z}{2} \partial_z+w\partial_w\\ V_2=\partial_w \\ V_3=-\frac{z}{2}\partial_z+f\partial_f\\ V_4=\partial_z \\ \end{array} $      & $\begin{array}{lcr} \mathcal{A}^1_8:V_2 \\ \mathcal{A}^2_8:V_3\\\mathcal{A}^3_8:V_2+V_4\\ \mathcal{A}^4_8:V_1-V_3  \end{array}$     & $  \begin{array}{lcr} \{z,f\}\\ \{w,fz^2\}\\ \{w-z,f\}\\ \{\frac{w}{z},fz\}\end{array} $ &$ \begin{array}{lcr}g'^2-gg''=0\\g^2(g'+2a)=0 \\g^2g'-ag'^2+agg''=0 \\g^2g'-as^2g'^2+2asgg'+\\ ag^2+as^2gg''=0\end{array} $ &
  \vspace*{2mm}
 \\

 $9$ & $  \begin{array}{lcr} V_1=w \partial_w\\ V_3=z\partial_z-f\partial_f \\V_2=\partial_w \\ V_4=\partial_z \\ \end{array} $      & $\begin{array}{lcr} \mathcal{A}^1_9:V_2+V_4\\ \mathcal{A}^2_9:V_2+V_3\end{array}$     & $  \begin{array}{lcr} \{w-z,f\}\\ \{w-\mbox{ln}z,zf\}\end{array} $ &$ \begin{array}{lcr}ag^2g'+g'^2-gg''=0 \\ ag^2g'+g'^2-gg''=0 \end{array} $
  \vspace*{2mm}\\
\cline{1-5}
\end{tabular} \\[5mm]
\end{center}
\end{small}
\end{table}

\begin{table}[h]
\begin{center}
{\small \textbf{Table 5:} Group invariant solutions of the  (($2$D) Rf) equation.}

\begin{tabular}{llllllllll}
\cline{1-3}
$\mathcal{A}^i_j$ &  {\small \mbox{ Invariant\ solution}} & {\small \mbox{Solution\  after\ substitiuting}}  \\
\cline{1-3}  \vspace*{2mm}
$\mathcal{A}^1_1$ & $c_2s^{c_1}$ & $c_2y^{c_1-1}e^{-c_1x}$
\\ \vspace*{2mm}
$\mathcal{A}^2_1$ & $-2s+c_1$ & $\frac{-2t+c_1}{y(\mbox{ln}y-x)^2}$
\\ \vspace*{2mm}
$\mathcal{A}^3_1$ & $\frac{1}{2c^2_1}(1-\mbox{tanh}(\frac{\mbox{ln}s-c_2}{2c_1})^2)$ & $\frac{t}{2c^2_1y}(1-\mbox{tanh}(\frac{\mbox{ln}y-x-c_2}{2c_1})^2)$
\\ \vspace*{2mm}
$\mathcal{A}^4_1$ & $\frac{c_1e^{c_1(s+c_2)}}{-1+e^{c_1(s+c_2)}}$ & $\frac{c_1e^{c_1(x+t-\mbox{ln}y+c_2)}}{-y+ye^{c_1(x+t-\mbox{ln}y+c_2)}}$
\\ \vspace*{2mm}
$\mathcal{A}^1_2$ &‌ $c_{1}$ & $\frac{c_1}{x}e^{y}$
\\ \vspace*{2mm}
$\mathcal{A}^2_2$ &‌ $c_1$ & $c_1e^{y}$
\\ \vspace*{2mm}
$\mathcal{A}^1_3$ &‌ $\frac{c_1(1+c_1)s^{c_1}}{-s^{c_1}(1+c_1-as)+a
c_1c_2(1+c_1)(as-1)^{c_1+1}}$ & $\frac{c_1(1+c_1)(\frac{t-\mbox{ln}x}{\mbox{ln}y-\mbox{ln}x})^{c_1}}{xy\mbox{ln}\frac{y}{x}(-(\frac{t-\mbox{ln}x}{\mbox{ln}y-\mbox{ln}x})^{c_1}(1+c_1-\frac{at-a\mbox{ln}x}{\mbox{ln}y-\mbox{ln}x})+ac_1c_2(1+c_1)(\frac{at-a\mbox{ln}x}{\mbox{ln}y-\mbox{ln}x}-1)^{c_1+1})}$
\\ \vspace*{2mm}
$\mathcal{A}^2_3$ & $c_2s^{c_1}$ & $\frac{c_2}{x^2}(\frac{y}{x})^{c_1}$
\\ \vspace*{2mm}
$\mathcal{A}^3_3$ & $c_1$ & $\frac{c_1}{yx}$
\\\vspace*{2mm}
$\mathcal{A}^1_4$ & $-2as+c_1$ & $\frac{-2at+c_1}{xy(4-4\mbox{ln}\frac{y}{x}+(\mbox{ln}\frac{y}{x})^2)}$
\\\vspace*{2mm}
$\mathcal{A}^2_4$ & $\frac{1}{2c_1s}(1-\mbox{tanh}(\frac{\mbox{ln}s+c_2}{2\sqrt{ac_1}})^2)$ & $\frac{t}{2c_1xy}(1-\mbox{tanh}(\frac{\mbox{ln}y-\mbox{ln}x+c_2}{2\sqrt{ac_1}})^2)$
\\\vspace*{2mm}
$\mathcal{A}^1_5$ & ‌ $\frac{c_1(1+c_1)s^{c_1}}{-s^{c_1}(1+c_1-s)+c_1c_2(1+c_1)(s-1)^{c_1+1}}$ & $\frac{c_1(1+c_1)(\frac{t-\mbox{ln}x}{y-\mbox{ln}x})^{c_1}}{x(y-\mbox{ln}x)(-(\frac{t-\mbox{ln}x}{y-\mbox{ln}x})^{c_1}(1+c_1-\frac{t-\mbox{ln}x}{y-\mbox{ln}x})+c_1c_2(1+c_1)(\frac{t-y}{y-\mbox{ln}x})^{c_1+1})}$
\\\vspace*{2mm}
$\mathcal{A}^2_5$ & $c_2e^{c_1s}$ & $\frac{c_2}{x}e^{c_1(y-\mbox{ln}x)}$
\\\vspace*{2mm}
$\mathcal{A}^3_5$ & $\frac{2c_1e^{c_1(s+c_2)}}{-1+e^{c_1(s+c_2)}}$  & $\frac{2c_1e^{c_1(t+y-2\mbox{ln}x+c_2)}}{-x+xe^{c_1(t+y-2\mbox{ln}x+c_2)}}$
\\\vspace*{2mm}
$\mathcal{A}^1_6$ & $-2s+c_1$ & $\frac{-2t+c_1}{x(y-\mbox{ln}x)^2}$
\\\vspace*{2mm}
$\mathcal{A}^2_6$ & $\frac{1}{2c^2_1}(1-\mbox{tanh}(\frac{s+c_2}{2c_1})^2)$ & $\frac{t}{2c^2_1x}(1-\mbox{tanh}(\frac{y-\mbox{ln}x+c_2}{2c_1})^2)$
\\\vspace*{2mm}
$\mathcal{A}^3_6$ & $\frac{c_1e^{c_1(s+c_2)}}{-1+e^{c_1(s+c_2)}}$  & $\frac{c_1e^{c_1(t-y+\mbox{ln}x+c_2)}}{-x+xe^{c_1(t-y+\mbox{ln}x+c_2)}}$
\\\vspace*{2mm}
$\mathcal{A}^1_7$  & $c_1$ & $\frac{c_1}{y}$
\\\vspace*{2mm}
$\mathcal{A}^2_7$ & $\frac{c_1e^{c_1(s+c_2)}}{-1+e^{c_1(s+c_2)}}$ & $\frac{c_1e^{c_1(\mbox{ln}y-\mbox{ln}x+t+c_2)}}{xy(-1+e^{c_1(\mbox{ln}y-\mbox{ln}x+t+c_2)})}$
\\\vspace*{2mm}
$\mathcal{A}^1_8$ &  $c_2e^{c_1s}$ & $c_2e^{c_1(y-ax)}$
\\\vspace*{2mm}
$\mathcal{A}^2_8$ & $-2as+c_1$ & $\frac{-2at+c_1}{(y-ax)^2}$
\\\vspace*{2mm}
$\mathcal{A}^3_8$ &  $\frac{c_1ae^{c_1(s+c_2)}}{-1+e^{c_1(s+c_2)}}$ & $\frac{c_1ae^{c_1(ax-y+t+c_2)}}{-1+e^{c_1(ax-y+t+c_2)}}$
\\\vspace*{2mm}
$\mathcal{A}^4_8$ & $\frac{-c^2_1ae^{\frac{c_1}{s}}}{e^{\frac{c_1}{s}}(c_1-s)-sac^2_1c_2}$ & $\frac{-c^2_1ae^{\frac{c_1(y-ax)}{t}}}{e^{\frac{c_1(y-ax)}{t}}(c_1y-c_1ax-t)-c^2_1c_2at}$
\\\vspace*{2mm}
$\mathcal{A}^1_9$ & $\frac{c_1e^{c_1(s+c_2)}}{1-ae^{c_1(s+c_2)}}$ & $\frac{c_1e^{c_1(t-x-\frac{y}{a}+c_2)}}{1-ae^{c_1(t-x-\frac{y}{a}+c_2)}}$
\\\vspace*{2mm}
$\mathcal{A}^2_9$ & $\frac{c_1e^{c_1(s+c_2)}}{1-ae^{c_1(s+c_2)}}$ & $\frac{c_1e^{c_1(t-\mbox{ln}x-\frac{y}{a}+c_2)}}{x-axe^{c_1(t-\mbox{ln}x-\frac{y}{a}+c_2)}}$
\\
\cline{1-3}
\end{tabular} \\[5mm]
\end{center}
\end{table}
\section*{Conclusion}
In this paper, by using the adjoint representation of the symmetry group on its Lie algebra, we have constructed an optimal system of one-dimensional subalgebras of the  two-dimensional Ricci flow equation. Moreover, we have obtained the similarity reduced equations for each element of optimal system as well as some group invariant solutions of two-dimensional Ricci flow equation.
\begin{small}

\end{small}
\end{document}